\documentclass[12pt]{amsart}
\usepackage{geometry}
\usepackage{graphicx}
\usepackage{amssymb}
\usepackage{amsmath}
\usepackage{epstopdf}
\newtheorem{theorem}{Theorem}[section]
\newtheorem{lemma}[theorem]{Lemma}
\newtheorem{remark}[theorem]{Remark}

\newtheorem{definition}[theorem]{Definition}

\newcommand{\GF}{\mathrm{GF}}

\newcommand{\AGammaL}{\mathrm{A} \Gamma \mathrm{L}}

\newcommand{\Sym}{\mathsf{Sym}}

\newcommand{\Cay}{\mathsf{Cay}}
\newcommand{\GPaley}{\mathsf{GPaley}}

\newcommand{\spa}{\mathsf{span}}
\newcommand{\cp}{\scriptstyle \square \textstyle}

\newcommand{\vl}{\, | \,}
\newcommand{\cpss}{\, \scriptstyle \square \textstyle}

\newcommand{\Aut}{\mathsf{Aut}\,}

\newcommand{\F}{\mathbb{F}}

\renewcommand\leq{\leqslant}
\renewcommand\geq{\geqslant}

\title{Generalised Paley graphs with a product structure}
\author{Geoffrey Pearce and Cheryl E Praeger}
\date{November 2016}
\begin{document}
\begin{abstract}
A graph is {\em Cartesian decomposable} if it is isomorphic to a Cartesian product of (more than one) strictly smaller graphs, each of which has more than one vertex and admits no such decomposition. These smaller graphs are called the {\em Cartesian-prime} factors of the Cartesian decomposition, and were shown, by Sabidussi and Vizing independently, to be uniquely determined up to isomorphism.  We characterise by their parameters those generalised Paley graphs which are Cartesian decomposable, and we prove that for such graphs, the Cartesian-prime factors are themselves smaller generalised Paley graphs. This generalises a result of Lim and the second author which deals with the case where all the Cartesian-prime factors are complete graphs.  These results contribute to the determination, by parameters, of generalised Paley graphs with automorphism groups larger than the 1-dimensional affine subgroups used to define them.
\end{abstract}
\maketitle




\section{Introduction}
Let  $p$ be a prime, $n$ a positive integer, and $\F_{p^n} = \GF(p^n)$ the finite field of order $p^n$. 
For a factorisation $p^n-1=dk$, such that if $p$ is odd then $k$ is even, 
we define the {\em generalised Paley graph} $\GPaley(p^{n},k)$ as the graph with vertex set $\F_{p^n}$, 
such that the edges are the pairs $\{x,y\}$ of vertices 
for which $x-y$ is a $d^{th}$-power in $\F_{p^n}^* = \F_{p^n} \backslash \{0\}$. The condition 
`$k$ is even if $p$ is odd' ensures that the adjacency relation is symmetric, defining an 
undirected graph.  If $d=2$ then $\GPaley(p^{n},k)$ is a Paley graph, named in honour
of Raymond Paley. (In the literature Paley's 1933 paper \cite{Pal} is often cited for
this construction. However that paper concerns a construction, based on finite fields, of what we now call 
Hadamard matrices; a paper exploring the origin of the name Paley 
graph is being prepared \cite{J}.) 
Some other generalised Paley graphs 
(in particular with $d=3$ or $4$)  were studied because of their graph theoretic properties and also their applications
in coding theory and to maps on compact surfaces. We make a few comments about these applications in Remark~\ref{apps}.

Generalised Paley graphs as defined above were introduced by Lim and the second author in \cite{LP}.   
From the definition it can easily be seen that $\GPaley(p^{n},k)$ admits as a subgroup of automorphisms the
group $G(p^n,k)$ generated by the translations $x\mapsto x+a$ ($a\in\F_{p^n}$), 
the $d^{th}$-power multiplications $x\mapsto x a^d$  ($a\in\F_{p^n}^*$), where $p^n-1=dk$,
and the field automorphisms $x\mapsto x^{p^i}$ ($i\leq n$). 
The group $G(p^n,k)$ is sometimes referred to as the `affine subgroup' as it is the intersection
of the full automorphism group of $\GPaley(p^{n},k)$ with the 1-dimensional affine group
$\AGammaL(1,p^n)$. Moreover, $G(p^n,k)$ acts transitively on the arcs of $\GPaley(p^n,k)$,
demonstrating that these graphs are arc-transitive.
Sometimes $G(p^n,k)$ is the full automorphism group 
(for example, if $k$ is a multiple of $(p^n-1)/(p-1)$, see \cite[Theorem 1.2]{LP}),
and \cite[Problem 1.5]{LP} asks for a characterisation of all values of $p,n,k,d$ for which this is the case. 

In particular, a larger automorphism group occurs if   $\GPaley(p^{n},k)$ 
is not connected. This is known to happen precisely when $k \vert (p^a-1)$ for some 
proper divisor $a$ of $n$, and in this case the connected components are all isomorphic 
to a smaller $\GPaley(p^{n'},k')$ for a certain subfield $\F_{p^{n'}}$ contained in $\F_{p^a}$,  
\cite[Theorem 2.2]{LP}.   We henceforth assume that $\GPaley(p^{n},k)$ is connected. 
This is equivalent to assuming that $k$ is a \emph{primitive divisor} of $p^n-1$, 
that is to say,  $k\vert (p^n-1)$ but $k \not\vert (p^a-1)$ for any $a < n$.  

Another family of generalised Paley graphs with larger automorphism groups  is 
characterised in \cite[Theorem 1.2(2)]{LP}: namely, $\GPaley(p^{n},k)$ is
isomorphic to a Hamming graph if and only if $k= b (p^{n/b}-1)$ for some divisor $b$ of $n$
such that $b>1$.
Here we characterise a family of generalised Paley graphs which properly contains 
the Hamming graphs: we determine precisely when a generalised Paley graph is 
Cartesian decomposable, as defined in Subsection~\ref{cpdefs}. 


\begin{theorem} \label{cpforward}
Let $p$ be a prime, $n$ a positive integer, $k$ a primitive divisor of $p^n-1$, and $\Gamma = 
\GPaley(p^n,k)$ $($ so $\Gamma$ is connected $)$. Then the following are equivalent.
\begin{enumerate}
\item[(a)]  $\Gamma$ is Cartesian decomposable;

\item[(b)] $k=bc$ such that $b > 1$, $b\vl n$, and $c$ is a primitive divisor of $p^{n/b} - 1$;

\item[(c)] $\Gamma \cong\ \cpss^b \ \Gamma_0$, where $\Gamma_0 = \GPaley(p^{n/b},c)$, with $b, c$ as in $(b)$.
\end{enumerate}
\end{theorem}
Indeed if $\GPaley(p^{n},k)$ is Cartesian decomposable, then 
all of its Cartesian-prime factors are isomorphic and are themselves smaller generalised Paley graphs, so that
the automorphism group is larger than $G(p^n,k)$, (see the discussion in Section~\ref{cpsect}).  
Hamming graphs correspond to the case where $c=p^{n/b}-1$, or equivalently, $\Gamma_0 = \GPaley(p^{n/b},p^{n/b}-1)$
is the complete graph $K_{p^{n/b}}$ on $p^{n/b}$ vertices. 
Theorem~\ref{cpforward} is proved in Section \ref{proofs}.  

We note that the example $\GPaley(81,20)$ given in \cite[Example 1.6]{LP} is not
Cartesian decomposable (since $20$ is not equal to $bc$ for any primitive divisor 
of $3^{4/b}-1$ for $b=2$ or $4$), and yet has automorphism group (of order $233,280$)
larger than the affine subgroup $G(81,20)$. Thus there is still more to be 
discovered before we have a complete solution to \cite[Problem 1.5]{LP}: 
a determination of all $p, n, k$ such that
$\Aut(\GPaley(p^n,k))=G(p^n,k)$.
   
\begin{remark}\label{apps} 
{\rm
Various combinatorial properties of the family of generalised 
Paley graphs
$\GPaley(p^n,k)$ have been studied in the literature, especially the cases
$\frac{p^n-1}{k}\in\{3,4\}$. For example, their
adjacency properties were studied in \cite{A, E}, while the graphs were exploited to obtain
improved lower bounds for Ramsay numbers in \cite{HWS, HWZ, KWH, WQHG}.
The latter work was, in turn, developed further in \cite{WSLX} where the authors use a different
generalisation of Paley graphs: in the case of graphs with a prime number of vertices, 
their graphs have edge sets which are unions of the edge sets of certain generalised 
Paley graphs studied here. 
The cliques and colourings of generalised Paley graphs were studied in \cite{SS} and, 
in particular, equality of the clique and covering number of 
$\GPaley(p^n,k)$ was shown to imply that the associated affine subgroup $G(p^n,k)$ 
is non-synchronising \cite[Theorem 5.2]{SS}. 
Generalised Paley graphs have also been investigated in connection with permutation decoding. 
Earlier work \cite{GK,KL} on codes derived from the row span of adjacency and 
incidence matrices of Paley graphs was extended in \cite{SL} for all generalised 
Paley graphs.   
Finally, Chapter 9.9.1 of the book \cite{JW} on graph embeddings in Riemann surfaces 
deals with generalised Paley maps - the underlying graphs are 
generalised Paley graphs. The authors show in \cite[Theorem 9.2]{JW} that, if $M$ is 
a regular map on a compact surface, then the automorphism group of $M$ acts primitively 
and faithfully on vertices if and only if $M$ is isomorphic to a generalised Paley map.
}
\end{remark}

\noindent
{\em Acknowledgement:} \quad We are grateful to Gareth Jones for helpful discussions about the origin of the name Paley graphs.
We are also grateful for the nudge he gave us to write up our work for publication.  The beginnings of this investigation go back to an undergraduate research project of the first author.

\section{The Cartesian product} \label{cpsect}

A graph $\Gamma$ consists of a set $V\Gamma$ of vertices and a subset $E\Gamma$ of unordered pairs of distinct vertices, 
called the edges of $\Gamma$. A graph is connected if for any two vertices $\alpha, \beta$, there exists a finite vertex sequence
$\alpha_0,\alpha_1,\dots,\alpha_r$ such that $\alpha_0=\alpha, \alpha_r=\beta$, and  $\{\alpha_{i-1},\alpha_i\}\in E\Gamma$ for $i=1,\dots,r$. 
An arc of $\Gamma$ is an ordered pair $(\alpha,\beta)$ such that $\{\alpha,\beta\}\in E\Gamma$; and $\Gamma$ is arc-transitive if its 
automorphism group acts transitively on arcs. Note that each $\GPaley(p^n,k)$ is arc-transitive since $G(p^n,k)$ is transitive on arcs.

\subsection{Definitions} \label{cpdefs}
For a set of graphs $\Gamma_1,\ldots,\Gamma_b$, the \emph{Cartesian product} $\Gamma_1\, \cp \ldots\, \cp\, \Gamma_b$ is the graph with vertex set $V\Gamma_1 \times \ldots \times V\Gamma_b$ and edges $\{(\alpha_1,\ldots,\alpha_b),(\beta_1,\ldots,\beta_b)\}$ whenever there exists $i$ such that $\{\alpha_i,\beta_i\} \in E\Gamma_i$, and $\alpha_j = \beta_j$ for all $j \neq i$.  The Cartesian product construction is both commutative and associative (up to isomorphism).  We write $\cp^b\, \Gamma_0$ to mean 
$\Gamma_0\, \cp \ldots\, \cp\,\Gamma_0$ where the factor graph $\Gamma_0$ occurs $b$ times.

A graph $\Gamma$ is said to be \emph{Cartesian decomposable} if  $\Gamma \cong\, \Gamma_1 \cp \ldots\, \cp\, \Gamma_b$ for some $b > 1$, 
such that each $\Gamma_i$ has at least two vertices; and $\Gamma$ is called \emph{Cartesian-prime} if $|V\Gamma|>1$ and no such decomposition exists. If $\Gamma$ is a finite graph with at least two vertices and $\Gamma$ is not Cartesian-prime, then clearly 
$\Gamma$ has at least one decomposition $\Gamma_1\, \cp \ldots\, \cp\, \Gamma_b$ with $b>1$ and each of the $\Gamma_i$ Cartesian-prime.  
Sabidussi \cite{S} and Vizing \cite{V} showed independently that, if $\Gamma$ is connected with at least two vertices and $\Gamma$ is not Cartesian-prime, then the Cartesian-prime `factors' $\Gamma_i$ are unique up to isomorphism and the order of the factors, see \cite[Theorem 6.6]{HIK}.
The condition that $\Gamma$ is connected is necessary, see  \cite[Theorem 6.2]{HIK}, and indeed when $\Gamma$ is connected, each of its Cartesian-prime factors $\Gamma_i$ is also connected. 
To facilitate our analysis we introduce the following `standard form' for our graphs.

\begin{definition} {\em
Let $\Delta = \Gamma_1\, \cp \ldots\, \cp\, \Gamma_b$.  We call $\Delta$ a {\em simple} Cartesian product if the following hold.
\begin{itemize}
\item[(i)] $b > 1$ and each $\Gamma_i$ is Cartesian-prime.
\item[(ii)]  For $1 \leq j < k \leq b$, if $\Gamma_j \cong \Gamma_k$ then $\Gamma_j = \Gamma_k$ (by which we mean that $V\Gamma_j = V\Gamma_k$ and $E\Gamma_j = E\Gamma_k$).
\end{itemize}
}
\end{definition}

Note that if $\Gamma$ is not Cartesian-prime, then there exists a simple Cartesian product $\Delta = \Gamma_1\, \cp \ldots\, \cp\, \Gamma_b$ such that $\Gamma \cong \Delta$.

\subsection{Automorphisms}
Let $\Delta = \Gamma_1\, \cp \ldots\, \cp\, \Gamma_b$ be a simple Cartesian product.  We identify two types of automorphisms of $\Delta$.  Firstly, for $1 \leq i \leq b$, any automorphism $g_i$ of $\Gamma_i$ induces an automorphism of $\Delta$ in the following action on vertices of $\Delta$\,:
$$
g_i : (\alpha_1,\ldots,\alpha_{i},\ldots,\alpha_{b}) \longmapsto (\alpha_1,\ldots,\alpha_{i}^{g_i},\ldots,\alpha_{b}).
$$
Thus we have $\Aut\Gamma_1 \times \ldots \times \Aut\Gamma_b \leq \Aut\Delta$.
Secondly, we may permute equal Cartesian factors (recall that, by assumption, the Cartesian factors are either non-isomorphic or equal).  Such an automorphism may be viewed as an element $\sigma$ of $\Sym(\{1,\ldots,b\})$ acting on vertices of $\Delta$ by 
$$
\sigma : (\alpha_1,\ldots,\alpha_b) \longmapsto (\alpha_{1\sigma^{-1}},\ldots,\alpha_{b\sigma^{-1}}).
$$
Such an element $\sigma$ is a well defined element of $\Sym(V\Delta)$, and if $\sigma\ne 1$ then $\sigma$ induces a non-trivial automorphism of $\Delta$ if and only if, for each $i$, $\Gamma_i = \Gamma_{i\sigma^{-1}}$ (note that this does not necessarily mean that $i = i^{\sigma^{-1}}$).  Let $W_\Delta$ denote the group of all such automorphisms $\sigma$ of $\Delta$.  

The group $\langle g_i, W_\Delta \vl g_i \in \Aut\Gamma_i,\ 1\leq i\leq b \rangle$ generated by all the automorphisms described above is equal to $(\Aut\Gamma_1 \times \ldots \times \Aut\Gamma_b) \rtimes W_\Delta$.

\subsection{An induced partition of the edge set} \label{inducedE}
Suppose that $\Delta = \Gamma_1\, \cp\, \ldots \cp\, \Gamma_b$, and let $1 \leq i \leq b$.  For 
a $(b-1)$-tuple 
$$
\bar{\delta}_i :=(\delta_1,\ldots,\delta_{i-1},\delta_{i+1},\ldots,\delta_b) \in V\Gamma_1 \times \ldots \times V\Gamma_{i-1} \times V\Gamma_{i+1} \times \ldots \times V\Gamma_b,
$$ 
define 
$$
E(i,\bar{\delta}_i) := \{ \{\bar{\alpha},\bar{\beta}\} \vl \bar{\alpha},\bar{\beta} \in V\Delta, \{\alpha_i,\beta_i\} \in E\Gamma_i, \alpha_j = \beta_j = \delta_j \; \mathrm{for} \; j \neq i\}.
$$
Let 
$${\mathcal E}_\Delta := \{E(i,\bar{\delta}_i) \vl 1 \leq i \leq b, \bar{\delta}_i \in V\Gamma_1 \times \ldots \times V\Gamma_{i-1} \times V\Gamma_{i+1} \times \ldots \times V\Gamma_b\}.$$ Then ${\mathcal E}_\Delta$ is a partition of $E\Delta$, and we call this the {\em induced Cartesian edge partition}.

Note that for each $E(i,\bar{\delta}_i) \in {\mathcal E}_\Delta$, the subgraph of $\Delta$ induced by $E(i,\bar{\delta}_i)$ is isomorphic to $\Gamma_i$.

\subsection{Some results pertaining to the Cartesian product}
The uniqueness, for a connected graph $\Gamma$, of its Cartesian decomposition with Cartesian-prime factors has
important consequences for its symmetry. The following result is essentially \cite[Theorem 6.10]{HIK},
with part (b) an immediate corollary.

\begin{theorem} \label{ImKlav}
Suppose that $\Delta = \Gamma_1\, \cp \ldots\, \cp\, \Gamma_b$ is a simple Cartesian product with (Cartesian-prime) factors $\Gamma_i$, 
and that $\Delta$ is connected.  Then
\begin{enumerate}
\item[(a)] $\Aut\Delta = (\Aut\Gamma_1 \times \ldots \times \Aut\Gamma_b) \rtimes W_\Delta$; and
\item[(b)] if  ${\mathcal E}_\Delta$ is the induced Cartesian edge partition defined in Section $\ref{inducedE}$, 
then $\Aut\Delta$ preserves ${\mathcal E}_\Delta$.
\end{enumerate}
\end{theorem}

As mentioned above the generalised Paley graphs are all arc-transitive, and this property forces their
Cartesian-prime factors to be isomorphic.

\begin{lemma} \label{arctranspower}
Let $\Delta = \Gamma_1\, \cp \ldots\, \cp\, \Gamma_b$ be a simple Cartesian product such that $\Delta$ is connected, and assume that $\Delta$ is arc-transitive.  Then there exists $\Gamma_0$ such that $\Gamma_i = \Gamma_0$ for all $i$.
\end{lemma}
\begin{proof}
Let ${\mathcal E}_\Delta$ be the induced Cartesian edge partition defined in Section \ref{inducedE}, and let $1 \leq j \leq b$.  Then ${\mathcal E}_\Delta$ at least one part $E(1,\bar{\delta}_1)$ with first entry $1$, and at least one part $E(j,\bar{\delta}_j)$ with first entry $j$.

By Theorem \ref{ImKlav}(b), $\Aut\Delta$ preserves the partition ${\mathcal E}_\Delta$, and since by assumption $\Delta$ is arc-transitive, $\Aut\Delta$ acts transitively on ${\mathcal E}_\Delta$.  Hence there exists $g \in \Aut\Delta$ with $E(1,\bar{\delta}_1)^g = E(j,\bar{\delta}_j)$.  The parts $E(1,\bar{\delta}_1)$ and $E(j,\bar{\delta}_j)$ induce subgraphs of $\Delta$ isomorphic to $\Gamma_1$ and $\Gamma_j$ respectively, and it follows that $\Gamma_1 \cong \Gamma_j$ (and in fact $\Gamma_1 = \Gamma_j$ as $\Delta$ is a simple Cartesian product).  As $j$ was arbitrary, the result now follows with $\Gamma_0 := \Gamma_1$.
\end{proof}

Cartesian decomposability is easily recognised for Cayley graphs:  
Let $G$ be a group and $S$ a subset of $G$, such that $1_G \not\in S$ and $S = S^{-1}$.  Then the \emph{Cayley graph} $\Cay(G,S)$ has vertex set $G$, and edges $\{x,y\}$ whenever $xy^{-1} \in S$. 
We note that a generalised Paley graph $\GPaley(p^n, k)$ can be viewed as a Cayley graph $\Cay(\F_{p^n}^+,S)$, where $\F_{p^n}^+$ is the additive group of $\F_{p^n}$, and where $S = \langle \xi^{(p^n-1)/k} \rangle$ for some primitive element $\xi$ of $\F_{p^n}$.

\begin{lemma} \label{cartpower}
Let $G$ be a group with subgroups $H_i$, for $1 \leq i \leq b$, such that $G = H_{1} \times \ldots \times H_{b}$.  For each $i$ let $S_i$ be a subset of $H_i$ such that $1 \not\in S_i$ and $S_i = S_i^{-1}$, and let $\Gamma_i= \Cay(H_{i},S_{i})$.  Let $S = \bigcup_{i=1}^{b} S_{i}$, and let $\Gamma = \Cay(G,S)$.  Then $\Gamma = \Gamma_1\, \cp \ldots\, \cp\, \Gamma_b$.
\end{lemma}

\begin{proof}
Note that $V\Gamma = G = H_1\times\dots\times H_b = V\Gamma_1\times\dots\times V\Gamma_b$.
Let $g=(g_1,\dots,g_b), h=(h_1,\dots,h_b)\in G$. Then $\{g,h\}\in E\Gamma$ if and only if $gh^{-1}\in S$
(by the definition of $\Cay(G,S)$), and this holds if and only if, for some $i$, $gh^{-1}\in S_i$ (since  
$S = \bigcup_{i=1}^{b} S_{i}$). This latter condition is equivalent to $g_ih_i^{-1}\in S_i$ and $g_j
=h_j$ for all $j\ne i$. Finally this holds if and only if $\{g,h\}$ is an edge of
$ \Gamma_1\, \cp \ldots\, \cp\, \Gamma_b$ (by the definition of the Cartesian product). 
\end{proof}

\section{Proof of the main result} \label{proofs}

\subsection{Preparation}\label{prep} 
We write $\F_{p^n}^+$ for the additive group of $\F_{p^n}$, and for a subset $U$ of $\F_{p^n}$ we write $\langle U \rangle^+$ for the additive subgroup of $\F_{p^n}$ generated by $U$.  Recall that for any subfield $\F_{p^{n/b}}$ (of order $p^{n/b})$, the field $\F_{p^n}$ has the structure of a $b$-dimensional vector space over $\F_{p^{n/b}}$.  In this context, we write $\spa_{p^{n/b}}(U)$ to denote the $\F_{p^{n/b}}$-span of $U$ as a subspace of $\F_{p^n}$.

Let $\xi$ be a primitive element of $\F_{p^n}$, let  $k$ be a primitive divisor of $p^n-1$, and let $\Gamma = \GPaley(p^n,k)$. Then, as discussed above, $\Gamma=\Cay(\F_{p^n}^+,S)$ with $S = \langle \xi^{(p^n-1)/k} \rangle$.  For each $u \in \F_{p^n}$, let $t_u$ denote the translation $t_u : \F_{p^n} \longrightarrow \F_{p^n} : x \longmapsto x + u$, and let $T := \{t_u \vl u \in \F_{p^n}\}$, the translation group of $\F_{p^n}$. Then $T \cong \F_{p^n}^+$.  For each $s \in S$, let $\hat{s}$ be the map $\hat{s} : \F_{p^n} \longrightarrow \F_{p^n} : x \longmapsto xs$, and let $\hat{S} := \{\hat{s} \vl s \in S\}$.  Then $G(p^n,k)=T \rtimes \hat{S}$ is the affine subgroup of $\Aut \Gamma$.

The following Lemma is proved in part (2) of the proof of \cite[Theorem 2.2]{LP}.
\begin{lemma} \label{PraLim222}
Let $\F$ be a finite field and let $S$ be a multiplicative subgroup of $\F^*$.  Then $\langle S \rangle^+$ is a subfield of $\F$.
\end{lemma}

\subsection{Proof of Theorem~\ref{cpforward}}
 From now on let  $p, n, k, \Gamma, S$  be as in Subsection~\ref{prep},
and let $d=(p^n-1)/k$.
Then $\Gamma$ is connected, by \cite[Theorem 2.2]{LP}. 

\begin{lemma}\label{bimpliesc} 
Suppose that $k=bc$ such that $b > 1$, $b\vl n$, and $c$ is a primitive divisor of $p^{n/b} - 1$.
Then $\Gamma \cong\ \cpss^b \ \Gamma_0$, where $\Gamma_0 = \GPaley(p^{n/b},c)$ is connected;
in particular $\Gamma$ is Cartesian decomposable.
\end{lemma}

\begin{proof}
By assumption, $k=|S|=bc$, so that $C:=\langle \xi^{\frac{p^n-1}{c}} \rangle$ is a 
subgroup of order $c$ of the multiplicative group $S$.  For $1 \leq i \leq b$, let 
$S_i = C\xi^{di}$, a multiplicative coset of $C$ in $S$. Observe that the $S_i$ 
are pairwise distinct, and $S_b=C$. 
Further, $S = \bigcup_{i=1}^b S_i$, so $\{S_1, \ldots, S_b\}$ is a partition of $S$.

As $\Gamma$ is connected, we have $\F_{p^n}^+ = \langle S \rangle^+$; and in particular, 
$\F_{p^n} = \spa_{p^{n/b}}(S)$.  Now $B = \{\xi^{di} \vl 1 \leq i \leq b\}$ is a set of 
coset representatives for $C$ in $S$.  Since $c \vl (p^{n/b}-1)$, $C$ is contained in 
the subfield $\F_{p^{n/b}}$, and 
it follows that $\spa_{p^{n/b}}(B) = \F_{p^n}$.  
In fact since $|B| = b$, $B$ is a basis for $\F_{p^n}$ as a vector space over $\F_{p^{n/b}}$.  
Thus $\F_{p^n}$ has a direct sum decomposition $\F_{p^n} = \bigoplus_{i=1}^b \spa_{p^{n/b}}(\xi^{di})$.

Now, for each $j$ with $1 \leq j \leq b$, we have $\langle S_j \rangle^+ \subseteq 
\spa_{p^{n/b}}(S_j)$.  As $S_j = C \xi^{dj}$ and $C \subseteq \F_{p^{n/b}}$, it follows 
that $\langle S_j \rangle^+ \subseteq \spa_{p^{n/b}}(\xi^{dj}) = \F_{p^{n/b}} \xi^{dj}$.  
Thus $|\langle S_j \rangle^+| \leq p^{n/b}$.

Let $1 \leq j \leq b$, and suppose, for a contradiction, that $\langle S_j \rangle^+ 
\neq \spa_{p^{n/b}}(\xi^{dj})$.  Then $|\langle S_j \rangle^+| < p^{n/b}$.  Since $\Gamma$ is connected,
$$
p^n = |\langle S \rangle^+| \leq |\bigoplus_{i=1}^b \langle S_i \rangle^+| \leq |\langle S_j \rangle^+|.(p^{n/b})^{b-1}.
$$
Since $|\langle S_j \rangle^+| < p^{n/b}$, the cardinality $|\langle S_j \rangle^+|.(p^{n/b})^{b-1} < p^n$; 
a contradiction.  Hence $\langle S_j \rangle^+ = \spa_{p^{n/b}}(\xi^{dj})$, and since $j$ 
was arbitrary and $\F_{p^n} =  \bigoplus_{i=1}^b \spa_{p^{n/b}}(\xi^{di})$, it follows that 
$\F_{p^n} = \bigoplus_{i=1}^b \langle S_i \rangle^+$.

Since $\langle C \rangle^+$ is contained in $\F_{p^{n/b}}$, and since $|\langle C \rangle^+| = p^{n/b}$ 
we have $\langle C \rangle^+ = \F_{p^{n/b}}$.  Since $C$ is a multiplicative subgroup of $\F_{p^{n/b}}^*$
of order $c$, the graph $\Cay(\F_{p^{n/b}}^+,C) = \GPaley(p^{n/b},c)$, and this graph is connected 
since $c$ is a primitive divisor of $p^{n/b}-1$.

As we showed above, for each $i$ the subset $S_i = C \xi^{di}$ and $\langle S_i \rangle^+ = 
\langle C \rangle^+ \xi^{di}$, and it is a straightforward consequence that the automorphism 
$\hat{\xi}^{di}$ of $\Gamma$ induces an isomorphism from $\Cay(\F_{p^{n/b}}^+,C)$ to 
$\Cay(\langle S_i \rangle^+, S_i)$.  Hence $\Cay(\langle S_i \rangle^+,S_i) \cong \GPaley(p^{n/b},c)$ 
for each $i$.  Since $\langle S \rangle^+ = \bigoplus_{i=1}^b \langle S_i \rangle^+$ (which is 
equivalent to $\langle S \rangle^+ = \langle S_1 \rangle^+ \times \ldots \times \langle S_b \rangle^+$), 
we may now apply Lemma \ref{cartpower} to obtain the result.
\end{proof}

We now prove the converse.

\begin{lemma}\label{aimpliesbc}
Suppose that $\Gamma=\GPaley(p^n,k)$ is Cartesian decomposable. Then
$k=bc$ where $b > 1$, $b\vl n$, and $c$ is a primitive divisor of $p^{n/b} - 1$.
\end{lemma} 

\begin{proof}
Since $\Gamma$ is arc-transitive, by Lemma \ref{arctranspower} there exists $b$ 
such that $\Gamma \cong \cp^b \Gamma_0$, and since $\Gamma$ is Cartesian decomposable,
we may assume that $b>1$ and that $\Gamma_0$ is Cartesian-prime.  Hence $|V\Gamma_0|^b = |V\Gamma| = p^n$, 
and so $b$ divides $n$ and $|V\Gamma_0| = p^{n/b}$.  Let $c$ be the valency of 
$\Gamma_0$.  Then the valency $k$ of $\Gamma$ is equal to $bc$, and it remains 
to show that $c$ is a primitive divisor of $p^{n/b}-1$.

Let $d = \frac{p^n-1}{k}$, let $S = \langle \xi^d \rangle$, and observe that $S$ consists of all the 
vertices of $\Gamma$ adjacent to $0$.  Let $\Delta = \cp^b \Gamma_0$, and let 
${\mathcal E}_\Delta$ be the induced Cartesian edge partition of $\Delta$, as 
defined in Section \ref{inducedE}.  Let $\varphi$ be an isomorphism from $\Delta$ 
to $\Gamma$, and let ${\mathcal P} = ({\mathcal E}_\Delta)\varphi$; that is, 
${\mathcal P} = \{(E)\varphi \vl E \in {\mathcal E}_\Delta\}$ where 
$(E)\varphi = \{ \{(\alpha)\varphi,(\beta)\varphi\} \vl \{\alpha,\beta\} \in E\}$ 
for all $E \in {\mathcal E}_\Delta$.  Thus ${\mathcal P}$ is the partition of $E\Gamma$ 
corresponding to ${\mathcal E}_\Delta$ under $\varphi$, and each part in ${\mathcal P}$ induces a subgraph 
of $\Gamma$ isomorphic to $\Gamma_0$. Let $P_b$ be the part in ${\mathcal P}$ containing 
the edge $\{0,\xi^{db}\}$, and let $S_b$ be the subset of $S$ consisting of vertices $\alpha$ such that $\{0,\alpha\} \in P_b$. Then $|S_b| = c$, the valency of $\Gamma_0$.

Since $\Gamma_0$ is Cartesian-prime, it follows from Theorem \ref{ImKlav} that 
$\Aut\Delta = \Aut\Gamma_0\wr S_k$ and preserves ${\mathcal E}_\Delta$. 
Hence $\Aut\Gamma$ must preserve ${\mathcal P}$. The subgroup $\hat{S} = 
\langle \hat{\xi}^d \rangle \leq \Aut\Gamma$ acts transitively on $S$, 
and since $\hat{S}$ preserves ${\mathcal P}$ and fixes $0$, $S_b$ is a 
block of imprimitivity  for the induced
group $\hat{S}^S$.  Since $\hat{S}^S$ is regular, this implies that $S_b$ 
is a coset of a multiplicative subgroup of $S$, and that the size $c$ of $S_b$ 
divides $|S|$.  Since $S$ is cyclic, there is a unique subgroup of order $c$, 
namely $\langle \xi^{\frac{p^n-1}{c}} \rangle$.  As $\xi^{\frac{p^n-1}{c}} = 
\xi^{db} \in S_b$ (by the definition of $S_b$), we have $S_b = \langle 
\xi^{\frac{p^n-1}{c}} \rangle$.  

Let $S_1,\ldots,S_{b-1}$ be the remaining $b-1$ cosets of $S_b$ in $S$; 
so for $1 \leq i \leq b$ there is an element $v_i \in S$ such that $S_i 
= S_bv_i$ (with $v_b = \xi^{db}$).  For each $i$, we have $\langle 
S_b v_i \rangle^+ = \langle S_b \rangle^+ v_i$, and it follows that 
$|\langle S_i \rangle^+| = |\langle S_b \rangle^+|$ for each $i$.  
The size of $\langle S \rangle^+$ is at most $|\bigoplus_{i=1}^b 
\langle S_i \rangle^+| = |\langle S_b \rangle^+|^b$.  As $\langle S 
\rangle^+ = \F_{p^n}$ (since $\Gamma$ is connected), it follows that 
$|\langle S_b \rangle^+| \geq p^{n/b}$.

Let $G=G(p^n,k) = T \rtimes \hat{S} \leq \Aut\Gamma$ as defined in 
Section \ref{prep}, and recall that $G$ acts transitively on edges and hence 
on ${\mathcal P}$. Noting that the valency of $\Gamma$ is $k=bc$, and that $P_b$ 
induces a subgraph with $p^{n/b}$ vertices and valency $c$, we obtain
$$
|{\mathcal P}| = \frac{|E\Gamma|}{|P_b|} = \frac{p^n k/2}{p^{n/b} c/2} = (p^{n-n/b}) b.
$$
As $T \triangleleft G$, the $T$-orbit $P_b^T$ is a block of imprimitivity for $G^{\mathcal P}$, 
and so $|P_b^T|$ divides $|{\mathcal P}|$. Since $T$ is a $p$-group it
follows by the `orbit-stabiliser theorem' that $|P_b^T|$ is a $p$-power, and 
since $b=k/c$ divides $p^n-1$, this implies that $|P_b^T|$ divides 
$p^{n-n/b}$. 
Next consider the set $V_b$
of vertices of $\Gamma$ incident with an edge of $P_b$. Then $(V_b)\varphi^{-1}$ 
is the set of vertices of $\Delta = \cp^b \Gamma_0$ incident with an edge 
of $(P_b)\varphi^{-1} \in {\mathcal E}_\Delta$. It follows from the definition 
of ${\mathcal E}_\Delta$ in Section~\ref{inducedE} that $|V_b| =|(V_b)\varphi^{-1}|
= |V\Gamma_0|$ which is $p^{n/b}$. Finally, since $T$ is transitive on 
$V\Gamma$ and $\Gamma$ is connected,  $ p^n = 
|V\Gamma| = |\bigcup_{t \in T} V_b^t|$. The union has size at most 
$|P_b^T|\times |V_b|\leq p^{n-n/b}\times p^{n/b} = p^n$. 
We must have equality, and this occurs only if $|P_b^T| = p^{n-n/b}$ 
and distinct images $V_b^t, V_b^{t'}$ are pairwise disjoint. 
Hence $V_b$ is a block of imprimitivity for $T$ acting on $V\Gamma$, 
which means that $V_b$ is an additive subgroup of $\F_{p^n}^+ = V\Gamma$.
Since the set $S_b \subset V_b$, the additive subgroup 
$\langle S_b \rangle^+ \leq V_b$. Then since $|V_b| = p^{n/b}$ and 
(as we showed earlier) $|\langle S_b \rangle^+| \geq p^{n/b}$, we obtain 
that $|\langle S_b \rangle^+| = p^{n/b}$, so $V_b = \langle S_b \rangle^+$. 
It now follows from Lemma \ref{PraLim222} that $V_b$ is the 
additive group of a subfield, namely the unique subfield $\F_{p^{n/b}}$
of order $p^{n/b}$. In particular, $S_b$ is a subgroup of the multiplicative group 
of $\F_{p^{n/b}}$, so $c\,|\,(p^{n/b}-1)$. The fact that $\langle S_b \rangle^+$ 
is equal to the full additive group of $\F_{p^{n/b}}$ implies that $S_b$ is 
contained in no proper subfield, and hence $c$ does not divide $p^a-1$ for any $a< n/b$. 
Thus $c$ is a primitive divisor of $p^{n/b} - 1$. 
\end{proof}

Theorem~\ref{cpforward} now follows from Lemmas~\ref{bimpliesc} and  ~\ref{aimpliesbc}.

\end{document}